\documentclass[11pt, reqno]{amsart}

\usepackage[top=3.75cm, bottom=3cm, left=3cm, right=3cm]{geometry}
\usepackage{amsmath}
\usepackage{amsthm}
\usepackage{amsfonts}
\usepackage{amssymb}
\usepackage{mathtools}
\usepackage{bm}
\usepackage[colorlinks=true, citecolor=citation, urlcolor=citation, linkcolor=reference]{hyperref}
\usepackage{color}
\usepackage[all,cmtip]{xy}
\usepackage{enumitem}
\usepackage{graphicx}
\usepackage{scalerel}
\usepackage{comment}
\usepackage[mathscr]{eucal}
\usepackage{stmaryrd}
\usepackage{tikz-cd}

\definecolor{citation}{rgb}{0,.40,.80}
\definecolor{reference}{rgb}{.80,0,.40}

\numberwithin{equation}{section}

\theoremstyle{plain}
\newtheorem{theorem}{Theorem}[section]
\newtheorem{lemma}[theorem]{Lemma}
\newtheorem{proposition}[theorem]{Proposition}

\theoremstyle{definition}
\newtheorem{definition}[theorem]{Definition}
\newtheorem{example}[theorem]{Example}
\newtheorem{remark}[theorem]{Remark}
\newtheorem{construction}[theorem]{Construction}

\setlist[itemize]{leftmargin=*, itemsep={2pt}}
\setlist[enumerate]{leftmargin=*, itemsep={2pt}}

\newcommand{\st}{\mid} 
\newcommand{\set}[1]{\left\{ \, #1 \, \right\}}

\newcommand{\Cat}{\mathrm{Cat}}

\newcommand{\ch}{\mathrm{ch}}
\newcommand{\Db}{\mathrm{D^b}}
\newcommand{\Dperf}{\mathrm{D}_{\perf}}

\newcommand{\Sch}{\mathrm{Sch}}

\newcommand{\Coh}{\mathrm{Coh}}

\newcommand{\Spec}{\mathrm{Spec}}
\newcommand{\Bl}{\mathrm{Bl}}
\DeclareMathOperator{\cyclicquotient}{\tilde{\mathrm{\mathbf{X}}}}

\DeclareMathOperator{\Pic}{Pic}

\DeclareMathOperator{\Hom}{Hom} 
\DeclareMathOperator{\Ext}{Ext}

\newcommand{\rank}{\mathrm{rank}}

\newcommand{\Knum}{\mathrm{K}_{\mathrm{num}}}

\newcommand{\perf}{\mathrm{perf}}

\newcommand{\Stab}{\mathrm{Stab}}

\newcommand{\on}{\operatorname}

\newcommand{\cO}{\mathcal{O}}
\newcommand{\cA}{\mathcal{A}}

\newcommand{\cC}{\mathscr{C}}
\newcommand{\cD}{\mathscr{D}}

\newcommand{\cL}{\mathcal{L}}
\newcommand{\cM}{\mathcal{M}}

\newcommand{\cZ}{\mathcal{Z}}

\newcommand{\rK}{\mathrm{K}}

\newcommand{\bC}{\mathbf{C}}
\newcommand{\bD}{\mathbf{D}}

\newcommand{\bH}{\mathbf{H}}

\newcommand{\bP}{\mathbf{P}}
\newcommand{\bQ}{\mathbf{Q}}
\newcommand{\bR}{\mathbf{R}}

\newcommand{\bX}{\mathbf{X}}

\newcommand{\bZ}{\mathbf{Z}}

\begin{document}

\title{Stability conditions on crepant resolutions of quotients of product varieties} 

\author{Alexander Perry}
\address{Department of Mathematics, University of Michigan, Ann Arbor, MI 48109 }
\email{arper@umich.edu}

\author{Saket Shah}
\address{Department of Mathematics, University of Michigan, Ann Arbor, MI 48109 }
\email{sakets@umich.edu}

\begin{abstract}
We construct stability conditions on crepant resolutions of certain quotients of product varieties, 
giving as a special case the first examples of stability conditions on strict Calabi--Yau varieties of arbitrary dimension. 
Along the way, we prove the crepant resolutions are derived equivalent to the corresponding quotient stacks, 
verifying an instance of a conjecture of Bondal and Orlov. 
\end{abstract}

\maketitle

%%%%%%%%%%%%%%%%%%%%%%%%%%%%%%%%%%%%%%%%%%%%%%%%%%%%%% 

\section{Introduction} 
Bridgeland \cite{bridgeland} introduced the notion of a stability condition on a triangulated category. 
Heuristically, it is the analog of an ample line bundle
in noncommutative algebraic geometry, in that it allows for the formulation of moduli spaces of semistable objects. 
This theory has led to many applications, including a description of groups of autoequivalences of K3 surfaces \cite{bridgeland-K3, bayer-bridgeland}, structure theorems for Donaldson--Thomas invariants \cite{toda, FT-DT}, and advances in our understanding of hyperk\"{a}hler geometry \cite{bayer-macri-nef, bayer-macri-mmp, stability-families, stability-GM}. 

In general, the problem of constructing stability conditions is notoriously difficult, 
although they are conjectured to exist on the bounded derived category of coherent sheaves $\Db(X)$ of any smooth projective variety $X$.  
When $\dim X = 2$ this was proved in \cite{bridgeland-K3, arcara-bertram}, but already when $\dim X = 3$ the existence of stability conditions is only known for special classes of varieties, like Fano threefolds \cite{Chunyi, Macri-et-al-Fano}, abelian threefolds and certain quotients thereof \cite{Piyaratne-abelian1, Piyaratne-abelian2, CY3-stability}, and quintic threefolds \cite{chunyi-quintic}. 
The case when $X$ is a strict Calabi--Yau variety --- meaning that $K_X = 0$ and $h^p(\cO_{X}) = 0$ for $0 < p < \dim X$, such as when $X$ is quintic threefold --- has played a particularly important role, 
in part due to motivations from physics. 

The aim of this paper is to construct stability conditions on crepant resolutions of certain quotient varieties, giving as a special case the first examples of stability conditions on strict Calabi--Yau varieties of dimension bigger than $3$. 
As an ingredient of independent interest, we also prove that these crepant resolutions are derived equivalent to the corresponding quotient stacks, verifying a stacky instance of a conjecture of Bondal and Orlov. 

Specifically, we work over 
the complex numbers and consider the following mild generalization of varieties constructed by Cynk and Hulek \cite{cynk-hulek}. 
Let $X_1, \dots, X_n$ be smooth varieties, each equipped with a $\bZ/2$-action whose fixed locus is a divisor. 
The group $(\bZ/2)^n$ acts on the product $X_1 \times \cdots \times X_n$ 
via the involution on each factor, and the quotient 
by the subgroup 
\begin{equation}
\label{Gn}
G_n = \set{ (a_1, \dots, a_n)  \, \vert \, \textstyle \sum a_i = 0 } \subset (\bZ/2)^n
\end{equation}
admits a crepant resolution (Proposition~\ref{chvar}), which we denote by 
\begin{equation*}
\cyclicquotient(X_1, \dots, X_n) \to (X_1 \times \cdots \times X_n)/G_n
\end{equation*} 
and call the \emph{iterated $\bZ/2$-quotient crepant resolution}. 
As the name suggests, $\cyclicquotient(X_1, \dots, X_n)$ is constructed inductively by crepantly resolving a $\bZ/2$-quotient of a product of two varieties at each step. 
On the other hand, the map from the quotient stack 
\begin{equation*}
[(X_1 \times \cdots \times X_n) / G_n] \to (X_1 \times \cdots \times X_n)/G_n
\end{equation*} 
can be viewed as a stacky crepant resolution. 
Our first result verifies a stacky instance of the Bondal--Orlov conjecture \cite{bondal-orlov} that all crepant resolutions of a variety with Gorenstein singularities are derived equivalent. 

\begin{theorem}
\label{theorem-derived-equivalence}
Let $X_1, \dots, X_n$ be smooth quasi-projective varieties, each equipped with a $\bZ/2$-action whose fixed locus is a divisor. 
Then there is a derived equivalence 
\begin{equation*} 
\Db(\cyclicquotient(X_1, \dots, X_n)) \simeq \Db([(X_1 \times \cdots \times X_n) /G_n]). 
\end{equation*} 
\end{theorem} 

To prove the theorem, we use the inductive description of $\cyclicquotient(X_1, \dots, X_n)$ as a crepant resolution of a $\bZ/2$-quotient. 
By interpreting $\Db([(X_1 \times \cdots \times X_n) /G_n])$ as the $G_n$-invariant category of $\Db(X_1 \times \cdots \times X_n)$ and using a higher-categorical argument, this reduces us to an instance of the stacky Bondal--Orlov conjecture for a cyclic quotient singularity, 
which has been treated in \cite{KPS} building on ideas from the derived McKay correspondence established in \cite{BKR}. 

When one of the $X_i$ admits a $\bZ/2$-invariant stability condition and the rest of the $X_i$ are curves,  
we prove the existence of stability conditions on the corresponding crepant resolution. 

\begin{theorem}
\label{theorem-stability-conditions} 
Let $X$ be a smooth projective variety equipped with a $\bZ/2$-action whose fixed locus is a divisor. 
Assume $X$ admits a stability condition $\sigma$ which is $\bZ/2$-invariant. 
Let $C_1, \dots, C_n$ be smooth projective curves each of which is equipped with a nontrivial involution. 
\begin{enumerate}
\item There is an associated stability condition $\tilde{\sigma}$ on $\cyclicquotient(X, C_1, \dots, C_n)$, which is numerical if $\sigma$ is so. 
\item \label{stability-conditions-semi-geometric} 
If the Albanese morphism of $X$ is finite, all $C_i$ have genus at least $1$, and $\sigma$ is numerical, 
then the skyscraper sheaves of points on $\cyclicquotient(X, C_1, \dots, C_n)$ are all semistable of the same phase with respect to $\tilde{\sigma}$. 
\end{enumerate} 
\end{theorem} 

\begin{remark}
On smooth projective varieties, there conjecturally exist stability conditions which are \emph{geometric} in the sense that skyscraper sheaves of points are all stable of the same phase. 
In the situation of~\eqref{stability-conditions-semi-geometric}, the stability conditions we construct on $\cyclicquotient(X, C_1, \dots, C_n)$ are nearly geometric, but not quite since the skyscraper sheaves of some points are strictly semistable. 
\end{remark} 

The stability condition $\tilde{\sigma}$ given by Theorem~\ref{theorem-stability-conditions} admits an explicit inductive description. 
The main ingredient in the construction 
is a result of Liu \cite{products}, which from a stability condition on a smooth projective variety produces one on the product with any smooth projective curve. 
In particular, applying this iteratively we obtain a stability condition on $X \times C_1 \times \cdots \times C_n$. 
Using results from \cite{polishchuk, inducing-stability}, we show this stability condition descends to one on the quotient stack $[(X \times C_1 \times \cdots \times C_n)/G_{n+1}]$, which by Theorem~\ref{theorem-derived-equivalence} is derived equivalent to the crepant resolution $\cyclicquotient(X, C_1, \dots, C_n)$. 
Finally, we deduce claim~\eqref{stability-conditions-semi-geometric} in Theorem~\ref{theorem-stability-conditions} using a recent result of Fu, Li, and Zhao \cite{FLZ}, which says that any stability condition on a variety with finite Albanese morphism is necessarily geometric. 

In the special cases of Theorem~\ref{theorem-stability-conditions} where either $X$ and the $C_i$ are elliptic curves, or when $X$ is a K3 surface and the $C_i$ are elliptic, 
we obtain the promised examples of stability conditions on strict Calabi--Yau varieties of arbitrary dimension. 

\begin{example}[Cynk--Hulek varieties] 
The standard stability condition on a smooth projective curve is fixed by any automorphism of the curve (Example~\ref{example-stability-curve}).  
Thus for $n \geq 2$ elliptic curves $E_1, \dots, E_n$ (equipped with the natural involution), Theorem~\ref{theorem-stability-conditions} gives numerical stability conditions on $\cyclicquotient(E_1, \dots, E_n)$ such that skyscraper sheaves of points are all semistable of the same phase. 
The variety $\cyclicquotient(E_1, \dots, E_n)$ is an $n$-dimensional simply connected strict Calabi--Yau variety, 
first introduced by Cynk and Hulek \cite{cynk-hulek} as a testing ground for modularity conjectures in arithmetic geometry.  
In the case $n = 3$, such Calabi--Yau threefolds were earlier considered by Borcea~\cite{borcea-CM} and by Vafa and Witten \cite{vafa-witten} from the perspectives of complex multiplication and mirror symmetry. 
\end{example}

\begin{example}[Borcea--Voisin varieties] 
Let $X$ be a K3 surface with an anti-symplectic involution. 
The fixed locus of the corresponding $\bZ/2$-action is necessarily a divisor in $X$, 
and $X$ admits a $\bZ/2$-invariant stability condition (Example~\ref{example-invariant-stability-surface}). 
In particular, if $E_1, \dots, E_n$ are elliptic curves, then 
Theorem~\ref{theorem-stability-conditions} gives numerical stability conditions on $\cyclicquotient(X, E_1, \dots, E_n)$. 
The variety $\cyclicquotient(X, E_1, \dots, E_n)$ is an $(n+2)$-dimensional strict Calabi--Yau variety, which 
in the case $n = 1$ was studied by Borcea~\cite{borcea} and Voisin~\cite{voisin} from the perspective of mirror symmetry. 
\end{example}

\begin{remark}
There are variants of the construction $\cyclicquotient(X_1, \dots, X_n)$ where the role of $\bZ/2$ is replaced by $\bZ/3$, $\bZ/4$, or $\bZ/6$ \cite{cynk-hulek, burek}. 
Although we do not investigate it in this paper, we expect that the analog of Theorem~\ref{theorem-derived-equivalence} holds in each of these settings; if so, then the arguments behind Theorem~\ref{theorem-stability-conditions} would also directly apply. 
In particular, if $E_1, \dots, E_n$ are elliptic curves which are all equipped with an automorphism of order $d$ where $d=3$, $4$, or $6$, 
then we would obtain stability conditions on the Calabi--Yau crepant resolution of the quotient of $E_1 \times \cdots \times E_n$ by 
the subgroup $\set{ (a_1, \dots, a_n)  \, \vert \, \sum a_i = 0 } \subset (\bZ/d)^n$. 
\end{remark} 

\subsection*{Related work} 
After the release of the first version of this paper, we learned that in unpublished work Yiran Cheng independently proved the existence of stability conditions on Cynk--Hulek varieties, by a similar argument to ours. 

\subsection*{Organization of the paper}
In \S\ref{section-crepant-resolutions} we describe the construction and geometry of iterated $\bZ/2$-quotient crepant resolutions. 
In \S\ref{section-bondal-orlov} we prove Theorem~\ref{theorem-derived-equivalence} after some categorical preliminaries. 
In \S\ref{section-stability} we study stability conditions on quotients of products of curves, and in particular prove Theorem~\ref{theorem-stability-conditions}. 

\subsection*{Conventions}
With the exception of \S\ref{section-enhanced-categories} and \S\ref{section-group-actions}, we work over the complex numbers to 
conform with the assumptions of several references which we cite. 
However, one may check that our results also hold over algebraically closed fields of suitable characteristic. 

All functors are derived. In particular, for a morphism $f \colon X \to Y$ we
write $f^*$ and $f_*$ for the derived pullback and pushforward, and for $E, F
\in \Db(X)$ we write $E \otimes F$ for their derived tensor product.

\subsection*{Acknowledgements}
We are grateful to Andrei C\u{a}ld\u{a}raru and Michel Van den Bergh for interesting discussions related to this work. 

During the preparation of this paper, the first author was partially supported by NSF grants DMS-2112747, DMS-2052750, and DMS-2143271, and a Sloan Research Fellowship, and 
the second author was partially supported by NSF grant DMS-2052750. Part of this work was also completed while the first author was in residence at the Simons Laufer Mathematical Sciences Institute in Spring 2024. 
	
%%%%%%%%%%%%%%%%%%%%%%%%%%%%%%%%%%%%%%%%%%%%%%%%%%%%%%	
	
	\section{Crepant resolutions} 
	\label{section-crepant-resolutions} 
			
The purpose of this section is to construct iterated $\bZ/2$-quotient crepant resolutions and discuss their geometry. 

\subsection{Construction} 
We will be interested in crepant resolutions of singular varieties of the following form. 
\begin{lemma}
\label{lemma-gorenstein}
Let $X_1, \dots, X_n$ be smooth varieties equipped with $\bZ/2$-actions whose fixed loci are divisors. 
Then the quotient $(X_1 \times \cdots \times X_n)/G_n$ is Gorenstein, where $G_n \subset (\bZ/2)^n$ is the subgroup defined in~\eqref{Gn}. 
\end{lemma} 

\begin{proof}
By a lemma of Cartan \cite{cartan}, the action of $\bZ/2$ on $X_i$ can be locally linearized, so we may reduce to the case where the $X_i$ are affine spaces and the group actions are linear. 
Then it follows from the criterion of \cite{watanabe} that 
the quotient $(X_1 \times \cdots \times X_n)/G_n$ is Gorenstein. 
\end{proof} 

We follow Cynk and Hulek \cite{cynk-hulek} (who considered a slightly less general situation) for the construction of crepant resolutions of varieties as in Lemma~\ref{lemma-gorenstein}. 
The key observation is as follows. 

	\begin{lemma}
		\label{induction}
		Let $X_1$ and $X_2$ be smooth varieties equipped with $\bZ/2$-actions whose fixed loci are divisors $D_1 \subset X_1$ and $D_2 \subset X_2$. 
		We consider $X_1 \times X_2$ with the product $\bZ/2 \times \bZ/2$-action, and the blowup $\Bl_{D_1 \times D_2} X_1 \times X_2$ with the induced $\bZ/2 \times \bZ/2$-action. 
		Let $G \subset \bZ/2 \times \bZ/2$ be the diagonal copy of $\bZ/2$. 
		Then there is a commutative diagram 
                        \begin{center}
            \begin{tikzcd}
		(\Bl_{D_1 \times D_2} X_1 \times X_2)/G \arrow[r, "\pi"] \ar{d}[swap]{q} & (X_1 \times X_2)/G \ar[d, "\bar{q}"] \\ 
		\Bl_{D_1 \times D_2} (X_1/(\bZ/2) \times X_2/(\bZ/2)) \ar[r, "\bar{\pi}"] & X_1/(\bZ/2) \times X_2/(\bZ/2). 
            \end{tikzcd}
            \end{center}
            such that: 
            \begin{enumerate}
            \item  $\pi$ is a crepant resolution which is equivariant for the induced $\bZ/2 \cong (\bZ/2 \times \bZ/2)/G$ actions on the source and target. 
            \item The fixed locus of the induced $\bZ/2$-action on $(\Bl_{D_1 \times D_2} X_1 \times X_2)/G$ is a divisor, isomorphic to 
            $D_1 \times X_2/(\bZ/2) \sqcup X_1/(\bZ/2) \times D_2$. 
            \item $q$ and $\bar{q}$ are the quotient morphisms for the induced actions of $\bZ/2$. 
            \end{enumerate} 

	\end{lemma}
        \begin{proof}
        The birationality and $\bZ/2$-equivariance of $\pi$ is clear from construction. 
        In particular, $\pi$ induces a map $\bar{\pi}$ between the $\bZ/2$-quotients of the source and target.
        
        Again by Cartan's lemma \cite{cartan}, the action of $G$ on $X_i$ can be locally linearized, so in particular the $D_i$ are smooth. It follows that $(\Bl_{D_1 \times D_2} X_1 \times X_2)/G$ is smooth, since it is a quotient of the smooth variety $\Bl_{D_1 \times D_2} X_1 \times X_2$ with fixed locus a smooth divisor, namely the exceptional divisor of the blowup.  
        Similarly, $X_1/(\bZ/2)$ and $X_2/(\bZ/2)$ are smooth, as is their blowup $\Bl_{D_1 \times D_2} (X_1/(\bZ/2) \times X_2/(\bZ/2))$ in the smooth center $D_1 \times D_2$. 

        Note that the morphism $q$ is a double cover, necessarily ramified along a divisor by purity of the branch locus. 
        Explicitly,  
        consider the copies of $D_1 \times X_2/(\bZ/2)$ and $X_1/(\bZ/2) \times D_2$ in $(X_1 \times X_2)/G$, given as the images of $D_1 \times X_2$ and $X_1 \times D_2$ under the quotient map. 
        Then the ramification locus of $q$ is given by the  
         strict transform in $(\Bl_{D_1 \times D_2} X_1 \times X_2)/G$ of their union, which is isomorphic to $D_1 \times X_2/(\bZ/2) \sqcup X_1/(\bZ/2) \times D_2$. 
         This proves the claim about the fixed locus of the $\bZ/2$-action on 
         $(\Bl_{D_1 \times D_2} X_1 \times X_2)/G$, since it coincides with the ramification locus of $q$.

        Finally, to check the crepancy of $\pi$, consider the diagram 
                        \begin{center}
            \begin{tikzcd}
                \on{Bl}_{D_1 \times D_2} X_1 \times X_2 \arrow[r, "g"] \arrow[d, "f"'] & X_1 \times X_2 \arrow[d] \\
                (\on{Bl}_{D_1 \times D_2} X_1 \times X_2)/G \arrow[r, "\pi"] & (X_1 \times X_2)/G
            \end{tikzcd}
            \end{center}
            Since $X_1 \times X_2 \to (X_1 \times X_2)/G$ is \'etale away from codimension 2, pullback preserves the canonical divisor. 
            Moreover, as $f$ is a double covering with nontrivial ramification along the exceptional divisor, $f^*$ is injective on $\Pic$.  
            Indeed, for any double covering $f: X \to Y$ branched along a divisor $D \subset Y$, there exists a line bundle $\cL$ on $Y$ such that $\cL^{\otimes 2} \simeq \cO_Y(D)$, and moreover $f_*\cO_X \simeq \cO_Y \oplus \cL^\vee$. If $f^*\cM \simeq \cO_X$ for any line bundle $\cM$, then one has 
            \[ \cM \oplus \cM \otimes \cL^\vee \simeq f_*f^*\cM \simeq \cO_Y \oplus \cL^\vee.\]
            Since splittings into indecomposables are unique in $\Coh(X)$, it follows that either $\cM \simeq \cO_Y$, as desired, or $\cM \simeq \cL \simeq \cL^\vee$, which is impossible as $\cL^{\otimes2} \simeq \cO_Y(D) \not\simeq\cO_Y$. 
            Thus, to conclude crepancy of $\pi$, it suffices to
check that the divisor $f^*K_{(\on{Bl}_{D_1 \times D_2} X_1 \times X_2)/G} - g^*K_{X_1 \times X_2}$ is trivial. But since $g$ is a blowup and $f$ is a double cover ramified along the exceptional divisor $E$, both $f^*K_{(\on{Bl}_{D_1 \times D_2} X_1 \times X_2)/G}$ and $g^*K_{X_1 \times X_2}$  equal $K_{\on{Bl}_{D_1 \times D_2} X_1 \times X_2} + E$.  
        \end{proof}
        
We inductively use Lemma~\ref{induction} to construct crepant resolutions of quotients of products with more factors.   
	\begin{proposition}
		\label{chvar}
		Let $X_1, \dots, X_n$ be smooth varieties equipped with $\bZ/2$-actions whose fixed loci are divisors $D_i \subset X_i$. 
		Then there exists a crepant resolution 
		\begin{equation*}
		 \pi \colon \cyclicquotient(X_1, \dots, X_n) \to (X_1 \times \cdots \times X_n)/G_n 
		\end{equation*} 
		where $\cyclicquotient(X_1, \dots, X_n)$ is equipped with a $\bZ/2$-action 
		such that:
		\begin{enumerate}
		\item 
		The fixed locus $\bD(X_1, \dots, X_n) \subset \cyclicquotient(X_1, \dots, X_n)$ of 
		the $\bZ/2$ action is a divisor, isomorphic to  
  \begin{equation*}
      (\bD(X_1, \dots, X_{n-1}) \times X_n/(\bZ/2)) \sqcup 
      (\cyclicquotient(X_1, \dots, X_{n-1})/(\bZ/2) \times D_n ) . 
  \end{equation*}
		\item The morphism $\pi$ is equivariant with respect to the $\bZ/2$-action on the source and the 
		induced action of $\bZ/2 \cong (\bZ/2)^n/G_n$ on the target. 
		\end{enumerate} 
		
	\end{proposition}
	\begin{proof}
For $n = 1$ we set $\cyclicquotient(X_1) = X_1$. 
For $n \geq 2$, by induction assume there is a crepant resolution 
\begin{equation*} 
\cyclicquotient(X_1, \dots, X_{n-1}) \to (X_1 \times \cdots \times X_{n-1})/G_{n-1}
\end{equation*}
with the desired properties. 
Taking the product with $X_n$ gives a crepant resolution 
\begin{equation*}
\cyclicquotient(X_1, \dots, X_{n-1}) \times X_n \to (X_1 \times \cdots \times X_{n-1})/G_{n-1} \times X_n
\end{equation*} 
which is equivariant with respect to the product action of $\bZ/2 \times \bZ/2$ on the source and target. 
Let $G \subset \bZ/2 \times \bZ/2$ be the diagonal copy of $\bZ/2$. 
Note that there is an isomorphism 
\begin{equation*}
\frac{(X_1 \times \cdots \times X_{n-1})/G_{n-1} \times  X_n}{G} \cong (X_1 \times \cdots \times X_n)/G_{n} 
\end{equation*} 
since the natural map $G_{n}/G_{n-1} \to (\bZ/2)^{n-1}/G_{n-1} \times \bZ/2\cong \bZ/2 \times \bZ/2$ is an isomorphism onto $G$. 
Thus passing to the quotient by $G$, the above crepant resolution induces a birational morphism 
\begin{equation*}
(\cyclicquotient(X_1, \dots, X_{n-1}) \times X_n)/G \to (X_1 \times \cdots \times X_n)/G_{n}, 
\end{equation*} 
which is equivariant with respect to the action of $\bZ/2 \cong (\bZ \times \bZ)/G$ on the source and $\bZ/2 \cong (\bZ/2)^n/G_n$ on the target. 
The canonical divisor pulls back to the canonical divisor under this map. Indeed, in the diagram 
\begin{center}
    \begin{tikzcd}
        \cyclicquotient(X_1, \dots, X_{n-1}) \times X_n \arrow[r, "f'"] \arrow[d, "g'"] & (X_1 \times \cdots \times X_{n-1})/G_{n-1} \times X_n \arrow[d, "g"] \\
        (\cyclicquotient(X_1, \dots, X_{n-1}) \times X_n)/G \arrow[r, "f"] & (X_1 \times \cdots \times X_n)/G_{n}
    \end{tikzcd}
\end{center}
the canonical map $f^*\omega_{(X_1 \times \cdots \times X_n)/G_{n}} \to \omega_{(\cyclicquotient(X_1, \dots, X_{n-1}) \times X_n)/G}$ pulls back under $g'$ to an isomorphism as the vertical maps $g$ and $g'$ are \'etale in codimension 1 and $f'$ is crepant, but pullback along $g'$ is conservative as $g'$ is faithfully flat.  
Therefore, we may take $\cyclicquotient(X_1,\dots,X_{n})$ to be the $\bZ/2$-equivariant crepant resolution 
\begin{equation*} 
\cyclicquotient(X_1,\dots,X_{n}) \to (\cyclicquotient(X_1,\dots,X_{n-1}) \times X_n)/G
\end{equation*} 
given by Lemma \ref{induction}.
	\end{proof}

We call the variety $\cyclicquotient(X_1, \dots,  X_n)$ constructed in Proposition~\ref{chvar} the \emph{iterated $\bZ/2$-quotient crepant resolution} of $(X_1 \times \cdots \times X_n)/G_n$.
        Such resolutions encompass a number of interesting classes of varieties, which have been studied especially in the Calabi--Yau case. 
        
        \begin{remark}
        As observed in \cite[\S2]{cynk-hulek}, when the $X_i$ are (strict) Calabi--Yau varieties, then so too is $\cyclicquotient(X_1, \dots, X_n)$. 
        \end{remark} 
        
        \begin{example}[Cynk--Hulek varieties] 
        \label{example-cynk-hulek}
        If $E_1, \dots, E_n$ are elliptic curves with their natural involutions, then 
        $\cyclicquotient(E_1,\dots,E_n)$ recovers the family of strict Calabi--Yau varieties constructed by Cynk and Hulek in \cite{cynk-hulek}. 
        \end{example} 

        \begin{example}
            As a variation on the same theme, one obtains non-Calabi--Yau varieties by taking $\cyclicquotient(C_1,C_2,\dots,C_n)$ for $C_i$ hyperelliptic curves with the natural involutions. 
        \end{example}
        
        \begin{example}[Borcea--Voisin threefolds] 
            \label{borcea-voisin}
            If $X$ is a K3 surface with an anti-symplectic involution and $E$ is an elliptic curve with its natural involution, then $\cyclicquotient(X,E)$ is a simply connected Calabi--Yau threefold which was introduced by Borcea and Voisin \cite{borcea, voisin}. 
        \end{example}

\subsection{Topology} 
	To end this section, we show that Cynk--Hulek varieties are simply connected, by adapting an argument of Spanier~\cite{spanier} for Kummer manifolds.
    \begin{proposition}
    \label{proposition-pi1}
        If $E_1, \dots, E_n$ are $n \geq 2$  
        elliptic curves, then $\cyclicquotient(E_1,\dots,E_n)$ is simply connected. 
    \end{proposition}
        
    \begin{proof}
        Let $\tilde{\bX}_n = \cyclicquotient(E_1,\dots,E_n)$ and let $\bD_n = \bD(E_1, \dots, E_n)$ be the fixed locus of the $\bZ/2$-action. 
          We first prove that $\bD_n$ has discrete fundamental groupoid $\pi_{\leq 1}(\bD_n)$.  By construction, we have 
    \begin{equation*}
             \bD_n = (\bD_{n-1} \times \bP^1) \sqcup 
      (\tilde{\bX}_{n-1}/(\bZ/2) \times E_n[2]),
      \end{equation*}   
      where $E_n[2] \subset E_n$ is the $2$-torsion in $E_n$. 
      It follows from Proposition~\ref{chvar} that the upon passing to $\bZ/2$-quotients, the resolution $\pi \colon \tilde{\bX}_{n-1} \to (E_1 \times \cdots \times E_{n-1})/G_n$ induces a morphism 
      \begin{equation*}
          \tilde{\bX}_{n-1}/(\bZ/2) \to (E/(\bZ/2))^{n-1} \cong (\bP^1)^{n-1}
      \end{equation*}
      which is a birational map of smooth projective varieties, and hence 
    \begin{equation*}
    \pi_1(\tilde{\bX}_{n-1}/(\bZ/2)) = \pi_1((\bP^1)^{n-1}) = 0.  
      \end{equation*} 
      The discreteness of $\pi_{\leq 1}(\bD_n)$ then follows by induction from the above decomposition of $\bD_n$. 
      
        Now we prove that $\tilde{\bX}_{n}$ is simply connected by induction on $n$. 
        By construction, $\tilde{\bX}_n$ is a resolution of $(\tilde{\bX}_{n-1} \times E_n)/(\bZ/2)$ where $\bZ/2$ acts diagonally. 
        Hence by \cite[Theorem 7.8]{kollar} we have 
        \begin{equation*}
 \pi_1(\tilde{\bX}_n) \cong \pi_1((\tilde{\bX}_{n-1} \times E_n)/(\bZ/2)), 
        \end{equation*}
        and so we may focus on proving triviality of the latter. 
        Consider the quotient map \[f : \tilde{\bX}_{n-1} \times E_n \to (\tilde{\bX}_{n-1}\times E_n)/(\bZ/2)\] which is an \'etale double cover away from the fixed locus $\bD_{n-1} \times E_n[2]$. Let 
        \[U = (\tilde{\bX}_{n-1} \times E_n) \setminus (\bD_{n-1} \times E_n[2])\]
        and let $V \subset \tilde{\bX}_{n-1} \times E_n$ be a $\bZ/2$-equivariant tubular neighborhood of $\bD_{n-1} \times E_n[2]$. 
        To compute the fundamental group of $(\tilde{\bX}_{n-1}\times E_n)/(\bZ/2)$, we will apply van Kampen's theorem to the cover by $f(U)$ and $f(V)$. 
        Since $f(V)$ is not connected, we will employ the groupoid version of this theorem; more specifically, we will use the criterion for simple connectivity from Lemma~\ref{lemma-van-kampen} below. For this purpose, we assemble some properties of the fundamental group(oids) of $f(U \cap V)$, $f(U)$, and $f(V)$; note that $U \cap V$ deformation retracts onto $\partial V$, so we may focus on $\partial V$ instead.

        Since the fixed locus $\bD_{n-1} \times E_n[2]$ is of codimension 2, there is a fiber bundle sequence 
        \[D^4 \to V \to \bD_{n-1} \times E_n[2], \]
        where $D^4$ is the unit 4-disk, while the boundary $\partial V$ similarly fits into a fiber bundle sequence
        \[S^3 \to \partial V \xrightarrow{\, g \,} \bD_{n-1} \times E_n[2].\]
        Thus, since we chose $V$ to be $\bZ/2$-equivariant, $f(V)$ fits into a fiber bundle sequence 
        \[c\mathbf{RP}^3 \to f(V) \to \bD_{n-1} \times E_n[2],\]
        where $c\mathbf{RP}^3$ is the contractible cone over $\mathbf{RP}^3$, while the boundary $\partial(f(V))$ fits into a fiber bundle sequence 
        \[\mathbf{RP}^3 \to \partial(f(V)) \to \bD_{n-1} \times E_n[2].\]
         By considering the associated long exact sequences on homotopy groups, we find from these fibration sequences that $\pi_{\leq 1}(f(V))$ is discrete (thanks to the discreteness of $\pi_{\leq 1}(\bD_{n-1})$ proved above), and the map $\pi_0(\partial(f(V))) \to \pi_0(f(V))$ is bijective. 

        Next we claim that the morphism 
         \[\pi_{\leq 1}(\partial(f(V))) \to \pi_{\leq 1}(f(U))\]
         is surjective on morphisms up to conjugation of paths. 
        Granting this, we conclude the simple connectivity of $\tilde{\bX}_{n}$ from Lemma~\ref{lemma-van-kampen} below with $U_1 = f(V)$ and $U_2 = f(U)$. 
        
        Since $U \to f(U)$ is an \'etale double cover, the complement $\tilde{\bX}_{n-1} \times E_n \setminus U$ is of codimension $2$, and $\pi_1(\tilde{\bX}_{n-1}) = 0$ by induction, there are isomorphisms 
        \[\pi_1(f(U)) \simeq \pi_1(U) \rtimes \bZ/2 \simeq \pi_1(\tilde{\bX}_{n-1} \times E_n) \rtimes \bZ/2 \simeq \bZ^2 \rtimes \bZ/2,\]
        where the semidirect product is induced by the natural $\bZ/2$-action on $\pi_1(\tilde{\bX}_{n-1} \times E_n) \simeq \bZ^2$ by negation.

        On the other hand, note that the 2-torsion $E_n[2]$ consists of 4 points which we will label as $p_{ij}$ for $i,j \in \{0,1\}$. Now fix any $x \in \bD_{n-1}$, and take $q_{ij} \in \partial V \subset \tilde{\bX}_{n-1} \times E$ to be any point in the fiber \[S^3 \cong g^{-1}(x,p_{ij}).\] 
        Let $a$ be a path from $-q_{00}$ to $q_{00}$ taken on the surface of the sphere $g^{-1}(x,p_{00})$, i.e. an antipodal path. Here, the negation denotes the action of $\bZ/2$ on the point $q_{00}$. Likewise, take $c_{ij}$ to be a path from $q_{ij}$ to $-q_{ij}$ on the sphere $g^{-1}(x,p_{ij})$. Take $\beta_{ij}$ to be a path in $U$ from $q_{00}$ to $q_{ij}$. \par 
        Now observe that the generators of $\pi_1(\tilde{\bX}_{n-1}\times E_n) \cong \bZ^2$ can be written as $a\beta_{10}c_{10}(-\beta_{10})^{-1}$ and $a\beta_{01}c_{01}(-\beta_{01})^{-1}$. After taking the quotient map $f$, note that $f(a) \in \pi_{\leq1}(\partial(f(V)))$ maps to the generator of $\bZ/2$ in $\pi_1(f(U))$, while we may write 
        \[f(a\beta_{10}c_{10}(-\beta_{10})^{-1}) = f(a)f(\beta_{10})f(c_{10})f(\beta_{10})^{-1}\]
        where $f(c_{10})$ is a loop also coming from $\partial(f(V))$. Similarly, one may write $f(a\beta_{01}c_{01}(-\beta_{01})^{-1})$ as coming from the image of $\pi_{\leq1}(\partial(f(V))) \to \pi_{\leq 1}(f(U))$ up to conjugation. This exhibits the necessary surjectivity. 
    \end{proof}

We invoked the following result above. 
    \begin{lemma}
    \label{lemma-van-kampen}
            Let $Y$ be a topological space with an open cover $Y = U_1 \cup U_2$ such that $\pi_{\leq1}(U_1)$ is discrete, $\pi_0(U_1 \cap U_2) \to \pi_0(U_1)$ is bijective, $U_2$ is connected, and the normal subgroupoid generated by the image of the map of groupoids \[\pi_{\leq1}(U_1 \cap U_2) \to \pi_{\leq1}(U_2)\]
            is $\pi_{\leq1}(U_2)$. 
            Then $Y$ is simply connected.
        \end{lemma}
        \begin{proof}
            This is an application of the groupoid version of van Kampen's theorem, one version of which can be found in \cite[Proposition 6.7.2]{groupoids}. \par 
            Recall that for any space $X$ and any collection of basepoints $A$ meeting every path-component of $X$, the fundamental groupoid $\pi_{\leq1}(X)$ may be presented  up to equivalence by a groupoid with vertices indexed by $A$: we denote this groupoid $\pi_{\leq1}(X,A)$. \par 
            Now let $A \subset U_1 \cap U_2$ be a set of representatives for $\pi_0(U_1 \cap U_2).$ By the groupoid van Kampen's theorem, it follows that 
            \begin{center}
                \begin{tikzcd}
                    \pi_{\leq1}(U_1 \cap U_2, A) \arrow[r] \arrow[d] & \pi_{\leq1}(U_1, A) \arrow[d] \\ 
                    \pi_{\leq1}(U_2, A) \arrow[r] & \pi_{\leq1}(Y,A)
                \end{tikzcd}
            \end{center}
            is a strict pushout of groupoids. Observe that this is a homotopy pushout in the canonical model structure on groupoids, since the top map is injective on vertices and hence cofibrant. \par 
            Now since $U_2$ is connected, $\pi_{\leq1}(U_2,A) \simeq \pi_1(U_2,y_0)$ as groupoids for any fixed $y_0 \in U_2$. Observe that the strict pushout of the resulting diagram 
            \begin{center}
                \begin{tikzcd}
                    \pi_{\leq1}(U_1 \cap U_2, A) \arrow[r] \arrow[d] & \pi_{\leq1}(U_1, A) \\ 
                    \pi_{1}(U_2, y_0) 
                \end{tikzcd}
            \end{center}
            is still a homotopy pushout as the top map is still cofibrant; as such, the pushout will be equivalent as a groupoid to $\pi_{\leq1}(Y,A).$ It therefore suffices to show that the pushout of this diagram is the trivial group, viewed as a groupoid. We verify this by checking the universal property. \par 
            Let $\cC$ be a cocone for this diagram, i.e. suppose that we have a commutative diagram 
            \begin{center}
                \begin{tikzcd}
                    \pi_{\leq1}(U_1 \cap U_2, A) \arrow[r] \arrow[d] & \pi_{\leq1}(U_1, A) \arrow[d] \\ 
                    \pi_{1}(U_2, y_0) \arrow[r] & \cC
                \end{tikzcd}
            \end{center}
            Since $A$ was chosen as a set of representatives for $\pi_0(U_1 \cap U_2) \cong \pi_0(U_1),$ it follows that $\pi_{\leq1}(U_1,A)$ is, as a groupoid, given by the collection of its vertices with only identity morphisms. Now the image of any vertex under the map $\pi_{\leq1}(U_1, A) \to \cC$ necessarily must be the image of the unique vertex of $\pi_1(U_2,y_0)$ by the commutativity of the above diagram; hence without loss of generality we may assume $\cC$ has a single vertex. Moreover, any morphism in $\pi_{\leq1}(U_1 \cap U_2, A)$ is sent to the identity morphism on the unique vertex of $\cC$, since the map $\pi_{\leq1}(U_1 \cap U_2, A) \to \cC$ factors through $\pi_{\leq1}(U_1, A)$ which has only identity morphisms. However, since the normal subgroupoid generated by $\pi_{\leq1}(U_1 \cap U_2,A)$ in $\pi_1(U_2,y_0)$ is all of $\pi_1(U_2,y_0)$ (i.e., $\pi_1(U_2,y_0)$ is generated by the image of $\pi_{\leq1}(U_1 \cap U_2,A)$ up to conjugation by loops in $\pi_1(U_2,y_0)$), it follows that any loop in $\pi_1(U_2,y_0)$ must be sent to the identity in $\cC$. It follows that the colimit of the diagram is the trivial group.
        \end{proof}
	
%%%%%%%%%%%%%%%%%%%%%%%%%%%%%%%%%%%%%%%%%%%%%%%%%%%%%%

\section{The Bondal--Orlov conjecture}	
\label{section-bondal-orlov} 

The goal of this section is to prove Theorem~\ref{theorem-derived-equivalence}, which can be regarded as a 
stacky instance of the Bondal--Orlov conjecture on derived equivalence of crepant resolutions of Gorenstein singularities. 
We begin with some preliminaries on enhanced categories and group actions. 
	
\subsection{Enhanced categories} 
 \label{section-enhanced-categories}
	It will be convenient to employ stable $\infty$-categories as enhancements of triangulated categories; useful references for the general theory are \cite{HTT, HA, SAG}. 
	For a scheme $X$, we denote by $\Dperf(X)$ and $\Db(X)$ the stable $\infty$-category incarnations of the category of perfect complexes and the bounded derived category of coherent sheaves; their homotopy categories recover the classical triangulated categories of such complexes. 

Fix a field $k$. 
There is also a theory of stable $\infty$-categories relative to $k$ (and more generally relative to any base scheme). 
We follow the conventions set out in \cite[Section 2]{NCHPD}. 
In particular, a \emph{$k$-linear category} is a small idempotent-complete stable $\infty$-category $\cC$ which is equipped with a $\Dperf(k)$-module structure. 
The collection of all $k$-linear categories is organized into an $\infty$-category 	which we denote by $\Cat_k$. 

	\begin{example}
			Let $X$ be a quasi-compact quasi-separated $k$-scheme. 
			Then $\Dperf(X)$ and $\Db(X)$ are $k$-linear categories, with the action of $V \in \Dperf(k)$ given by $\pi^*V \otimes -$ where $\pi \colon X \to \Spec(k)$ is 
			the structure morphism. Note that when $X$ is smooth, $\Dperf(X) = \Db(X)$. 
	\end{example}
	
	The category $\Cat_k$ admits a closed symmetric monoidal structure with unit $\Dperf(k)$ and tensor product which we denote by $\cC \otimes \cD$ 
	for $\cC, \cD \in \Cat_k$ (cf. \cite[Theorem 4.5.2.1]{HA}). 
	Moreover, $\Cat_k$ admits all limits and colimits (cf. \cite[\S2.1]{mathew-descent}). 
	
	\begin{lemma}
	\label{lemma-tensor-limit}
    If $\cC \in \Cat_k$ is smooth and proper, then the
	tensor product $- \otimes \cC \colon \Cat_k \to \Cat_k$ commutes with all limits. In particular, if $X$ is a smooth proper $k$-variety, then $- \otimes \Dperf(X)$ commutes with all limits. 
	\end{lemma}
	
	\begin{proof}  
    We refer to \cite[\S4]{NCHPD} for background on smooth proper categories. 
    In particular, by \cite[Lemma 4.8]{NCHPD}, $\cC$ is smooth and proper if and only if $\cC$ is dualizable as an object of the symmetric monoidal $\infty$-category $\Cat_k$. 
    In this case, the functor $- \otimes \cC$ is simultaneously left and right adjoint to the tensor product with its dual $- \otimes \cC^\vee$; in particular, by virtue of being a right adjoint, it commutes with all limits. 
    Finally, we note that if $X$ is a smooth proper $k$-variety, then $\Dperf(X)$ is indeed smooth and proper, by \cite[Lemma 4.9]{NCHPD}. 
	\end{proof} 
	
	The tensor product in $\Cat_k$ is compatible with geometric fiber products.  
	
	\begin{theorem}[{\cite[Theorem 1.2]{bzfn}}]
		\label{tensor}
		Let $X$ and $Y$ be quasi-compact quasi-separated $k$-schemes. 
		Then external tensor product induces an equivalence
\begin{equation*}
\boxtimes : \Dperf(X) \otimes \Dperf(Y) \xrightarrow{\, \sim \,} \Dperf(X \times Y). 
\end{equation*}
	\end{theorem}
	We note that the equivalence in the theorem above is natural in $X$ and $Y$, i.e. the external tensor product yields a natural equivalence of functors $\Sch^{\mathrm{qcqs}}_k \times \Sch^{\mathrm{qcqs}}_k \to \Cat_k$, where $\Sch^{\mathrm{qcqs}}_k$ is the category of quasi-compact quasi-separated $k$-schemes. This follows from the fact that the external tensor product is induced by applying the universal property of the tensor product in $\Cat_k$ to the functor 
    \[\Dperf(X) \times D_\perf(Y) \to D_\perf(X \times Y),\]
    given by pullback in each factor, which is natural in $X$ and $Y$. 
\subsection{Group actions} 	
\label{section-group-actions}
	Fix a finite discrete group $G$ and a field $k$. 
	Following the exposition in \cite{HH}, we discuss $G$-actions on a category $\cC \in \Cat_k$. 
	
	Let $BG$ denote the group $G$ viewed as a $1$-category with a single object $*$ whose endomorphisms are given by $\Hom_{BG}(*,*) = G$. By the nerve construction, we may view $BG$ as an $\infty$-category. 
	
	\begin{definition}
		A \emph{$G$-action on a category} $\cC \in \Cat_k$ is a functor $\rho: BG \to \Cat_k$ sending the unique 0-simplex $*$ of $BG$ to $\cC$. Viewing $\rho$ as a diagram of categories, we define the \textit{$G$-invariants} as $\cC^G := \lim(\rho)$. 
	\end{definition}

The category $\cC^G$ is well-defined, since as noted above $\Cat_k$ admits all limits. By taking the colimit along $\rho$, one may also define the $G$-coinvariants $\cC_G$, but we will not need this construction. 
	
	\begin{example}
	\label{example-quotient-stack}
		A $G$-action on a $k$-scheme $X$ corresponds to a functor $BG \to \Sch_k$ sending $*$ to $X$. Composing with the functor $\Dperf(-) \colon \Sch_k \to \Cat_k$ gives a $G$-action $\rho \colon BG \to \Cat_k$ on $\Dperf(X)$. It follows by descent that there is an equivalence $\Dperf(X)^G \simeq \Dperf([X/G])$, where $[X/G]$ is the quotient stack. 	
	\end{example}
	
	By using the following lemma, invariants can be computed successively along normal subgroups and their quotients. 

	\begin{lemma}
		\label{trans}
		Let 
		\[1 \to H \to G \to G/H \to 1\]
		be a short exact sequence of groups, and suppose that $\cC \in \Cat_k$ (or in any other complete $\infty$-category) admits a $G$-action $\rho : BG \to \Cat_k$. Then under the natural $H$-action $\rho|_{BH}$, the category of invariants $\cC^{H}$ admits a natural $(G/H)$-action such that 
		\[\cC^G \simeq (\cC^H)^{G/H}.\] 
	\end{lemma}
	\begin{proof}
		We have the fiber sequence 
		\begin{center}
			\begin{tikzcd}
				BH \arrow[r] \arrow[d] & BG \arrow[d, "f"] \\ 
				* \arrow[r] & B(G/H)
			\end{tikzcd}
		\end{center}
		where $f$ is a Cartesian fibration of $\infty$-categories (since $G \to H$ is a surjection of groups). 
		Now define \[f_*(\rho): B(G/H) \to \Cat_k\] to be the right Kan extension of $\rho$ along $f$ (which exists since $\Cat_k$ is complete). 
  		The diagram 
		\begin{center}
			\begin{tikzcd}
				BG \arrow[rr] \arrow[rd, "f"'] && B(G/H) \arrow[ld, "\mathrm{id}"] \\
				& B(G/H) &
			\end{tikzcd}
		\end{center}
		satisfies the hypotheses of (the opposite version of) \cite[Proposition 4.3.3.10]{HTT} since all morphisms in $BG$ are $f$-Cartesian and all morphisms in $B(G/H)$ are $\mathrm{id}$-Cartesian. 
		This forces that $f_*(\rho)(*) \simeq \lim \rho|_{BH} \simeq \cC^{H}$. Hence $f_*(\rho)$ defines a canonical $G/H$-action on $\cC^H$, and since limits may be computed after right Kan extension, we conclude  
		\[
		\cC^G = \lim \rho \simeq \lim f_*(\rho) \simeq (\cC^H)^{G/H}. \qedhere
		\]
	\end{proof}
	
	\begin{remark}
	The analog of Lemma~\ref{trans} in the setting of non-enhanced categories was proved in \cite[Proposition~3.3]{BO-equivariant-cats}, but we will need the enhanced version for our arguments. 
	In fact, as mentioned in the statement of Lemma~\ref{trans}, our result applies to any complete $\infty$-category $\cD$ in place of $\Cat_k$; 
	in particular, taking $\cD$ to be the category of non-enhanced $k$-linear categories, we recover \cite[Proposition~3.3]{BO-equivariant-cats}. 
	\end{remark}

\subsection{Proof of Theorem~\ref{theorem-derived-equivalence}} 
	The rest of this section is dedicated to the proof of Theorem~\ref{theorem-derived-equivalence}. 
	The basic strategy is to use the inductive definition of $\cyclicquotient(X_1, \dots, X_n)$ and Lemma~\ref{trans} to reduce to the case $n = 2$, where we have a crepant resolution of a cyclic quotient singularity, in which case we may appeal to the following result of Krug, Ploog, and Sosna. 

	\begin{theorem}[\cite{KPS}]
		\label{cyclic}
		Let $X$ be a smooth quasi-projective variety with an action by a cyclic group $G = \bZ/m$. Let $S \subset X$ be the fixed locus, and assume that $\mathrm{codim}_X(S) = m$ and a generator $g \in G$ acts on the normal bundle $N_{S/X}$ by multiplication by a primitive $m$-th root of unity. 
		Then in the diagram 
		\begin{center}
			\begin{tikzcd}
				\on{Bl}_S X \arrow[r, "p"] \arrow[swap, d, "q"] & X \arrow[d] \\
				(\on{Bl}_S X)/G \arrow[r] & X/G
			\end{tikzcd}
		\end{center}
		there is an isomorphism between $\tilde{Y} = (\on{Bl}_S X)/G$ and the fine moduli space $\on{Hilb}^G(X)$ of $G$-clusters on $X$, where $q : \on{Bl}_SX \to \tilde{Y}$ is the universal family. Moreover, the functor 
		\begin{equation}
		\label{kps-equivalence}
		\Dperf(\tilde{Y}) \xrightarrow{\mathrm{triv}} \Dperf(\tilde{Y})^G \xrightarrow{q^*} \Dperf(\on{Bl}_SX)^G \xrightarrow{p_*} \Dperf(X)^G,
		\end{equation} 
		where $\on{triv}$ gives each object the trivial $G$-equivariant structure, is an equivalence. 
	\end{theorem}

	For our inductive argument, it will be useful to prove the following slightly more precise version of Theorem~\ref{theorem-derived-equivalence}, which keeps track of the $\bZ/2$-action in play. 
	
	\begin{theorem}
	\label{theorem-kps}
	Let $X_1, \dots, X_n$ be smooth quasi-projective varieties, each equipped with a $\bZ/2$-action whose fixed locus is a divisor. 
Then there is a $\bZ/2$-equivariant equivalence 
\begin{equation*} 
\Dperf(\cyclicquotient(X_1, \dots, X_n)) \simeq \Dperf([X_1 \times \cdots \times X_n/G_n])
\end{equation*} 
where $\bZ/2$ acts on the left through its action on $\cyclicquotient(X_1, \dots, X_n)$ given by Proposition~\ref{chvar} and on the right through the natural action of $\bZ/2 \cong (\bZ/2)^{n}/G_n$ on $[X_1 \times \cdots \times X_n/G_n]$. 
	\end{theorem} 
	
	\begin{proof}
	Throughout this proof, it will be convenient to identify the derived category of a quotient stack with the corresponding invariant category, as in Example~\ref{example-quotient-stack}.
	
	We apply Theorem~\ref{theorem-kps} with 
	$X = \cyclicquotient(X_1, \dots, X_{n-1}) \times X_{n}$ and $G  \subset \bZ/2 \times \bZ/2$ the diagonal subgroup. 
	Then the fixed locus is the codimension $2$ subvariety $S = \bD(X_1, \dots, X_{n-1}) \times D_n$ where $\bD(X_1, \dots, X_{n-1})$ and $D_n$ are 
	the fixed loci of the $\bZ/2$-actions on each factor, and $\tilde{Y} =  \cyclicquotient(X_1, \dots, X_{n})$ by the construction of Proposition~\ref{chvar}.  
	We obtain an equivalence 
	\begin{equation}
	\label{equivalence-induction}
	\Dperf(\cyclicquotient(X_1, \dots, X_{n})) \simeq \Dperf(\cyclicquotient(X_1, \dots, X_{n-1}) \times X_{n})^G.
	\end{equation} 
	Note that $\bZ/2$ acts on the left through its action on $\cyclicquotient(X_1, \dots, X_{n})$, and on right by virtue of the fact that 
	$\bZ/2 \cong (\bZ/2 \times \bZ/2)/G$; we claim the equivalence is equivariant for these actions. 
	Indeed, the diagram 
			\begin{center}
			\begin{tikzcd}
				\Bl_{\bD(X_1, \dots, X_{n-1}) \times D_n}(\cyclicquotient(X_1, \dots, X_{n-1}) \times X_{n}) 
	 \arrow[r, "p"] \arrow[swap, d, "q"] & \cyclicquotient(X_1, \dots, X_{n-1}) \times X_n \arrow[d] \\ 
	 		(\Bl_{\bD(X_1, \dots, X_{n-1}) \times D_n} \cyclicquotient(X_1, \dots, X_{n-1}) \times X_{n})/G \arrow[r] & (\cyclicquotient(X_1, \dots, X_{n-1}) \times X_{n})/G
			\end{tikzcd}
		\end{center}
	is $(\bZ/2)^2$-equivariant, where $ \cyclicquotient(X_1, \dots, X_{n-1}) \times X_n$ is equipped with the product $\bZ/2$-action 
	and the other varieties are equipped with the induced $(\bZ/2)^2$-actions, which on the bottom row factor via the quotient 
	$(\bZ/2)^2 \to \bZ/2$ by the diagonal subgroup $G$. 
	Then the claimed $\bZ/2$-equivariance follows from the form of the equivalence~\eqref{kps-equivalence}.   
	
	Now we prove the result by induction on $n$. We compute 
		\begin{align*}
			\Dperf(\cyclicquotient(X_1,\dots,X_{n-1}) \times X_n) &\simeq \Dperf(\cyclicquotient(X_1,\dots,X_{n-1})) \otimes \Dperf(X_n)\\
			&\simeq \Dperf(X_1 \times \cdots \times X_{n-1})^{G_{n-1}} \otimes \Dperf(X_n) \\
			&\simeq \left(\Dperf(X_1 \times \cdots \times X_{n-1}) \otimes \Dperf(X_n)\right)^{G_{n-1}} \\
			&\simeq \Dperf(X_1 \times X_2 \times \cdots \times X_n)^{G_{n-1}},
		\end{align*}
where the first and fourth equivalences follow from Theorem~\ref{tensor}, the second follows from induction, the third follows from the commutativity of tensor product and limits (Lemma~\ref{lemma-tensor-limit}), and where each isomorphism is naturally equivariant with respect to the $(\bZ/2)^2$-actions on each category. 
To be more explicit about the action on the final line, by Lemma~\ref{trans} there is natural action by 
$(\bZ/2)^{n}/G_{n-1} = (\bZ/2)^{n-1}/G_{n-1} \times \bZ/2 \cong \bZ/2 \times \bZ/2$, and under this identification 
$G_{n}/G_{n-1} \subset (\bZ/2)^{n}/G_{n-1}$ corresponds to the diagonal subgroup $G \subset \bZ/2 \times \bZ/2$. 
Therefore, passing to $G$-invariants, the above equivalence gives a $\bZ/2$-equivariant equivalence 
\begin{align*}
\Dperf(\cyclicquotient(X_1,\dots,X_{n-1}) \times X_n)^G & \simeq ( \Dperf(X_1 \times X_2 \times \cdots \times X_n)^{G_{n-1}})^{G_{n}/G_{n-1}} \\ 
& \simeq \Dperf(X_1 \times X_2 \times \cdots \times X_n)^{G_n}
\end{align*} 
where the final line is given by Lemma~\ref{trans}. 
Combined with the $\bZ/2$-equivariant 
equivalence~\eqref{equivalence-induction}, this finishes the proof of the theorem. 
\end{proof} 

\begin{remark}
\label{remark-Phi-skyscraper}
Let $\Phi \colon \Dperf(\cyclicquotient(X_1, \dots, X_n)) \xrightarrow{\sim} \Dperf([X_1 \times \cdots \times X_n/G_n])$ be the equivalence provided by 
Theorem~\ref{theorem-kps}, and let 
$\pi \colon X_1 \times \cdots \times X_n \to [X_1 \times \cdots \times X_n/G]$ be the quotient map. 
For later use, we note that if $p \in \cyclicquotient(X_1, \dots, X_n)$ is a closed point, then $\pi^*\Phi(\cO_p)$ is a sheaf which admits a filtration whose graded pieces are structure sheaves of points. 
Indeed, first note that applying the equivalence~\eqref{equivalence-induction} followed by the forgetful functor 
\begin{equation*}
    \Dperf(\cyclicquotient(X_1, \dots, X_{n-1}) \times X_{n})^G \to \Dperf(\cyclicquotient(X_1, \dots, X_{n-1}) \times X_{n})
\end{equation*}
to $\cO_p$ gives the structure sheaf of a $G$-cluster on $\cyclicquotient(X_1, \dots, X_{n-1}) \times X_{n}$, which in particular admits a filtration whose graded pieces are structure sheaves of points. 
Now the claim follows by induction from the construction of $\Phi$ in the proof of Theorem~\ref{theorem-kps}. 
\end{remark}

%%%%%%%%%%%%%%%%%%%%%%%%%%%%%%%%%%%%%%%%%%%%%%%%%%%%%%	

\section{Stability conditions} 
\label{section-stability}

The goal of this section is to prove  Theorem~\ref{theorem-stability-conditions}. 
We start by reviewing general procedures for inducing stability conditions on quotient stacks and on products with curves. 

\subsection{Stability conditions and group actions}  
\label{section-stability-group-actions}
\begin{definition}
Let $\cD$ be a triangulated category, and let $v \colon \rK_0(\cD) \to \Lambda$ be a homomorphism from the Grothendieck group of $\cD$ to an abelian group $\Lambda$. 
A \emph{pre-stability condition on $\cD$ with respect to $v$} (or with respect to $\Lambda$ when $v$ is clear from context) 
is a pair $\sigma = (\cA, Z)$, where $\cA \subset \cD$ is the heart of a bounded t-structure and $Z \colon \Lambda \to \bC$ is a homomorphism satisfying the following conditions: 
\begin{enumerate}
\item \label{stability-function} For any $0 \neq E \in \cA$ we have $Z(v(E)) \in \bH \sqcup \bR_{< 0}$, 
where $\bH = \set{ z \in \bC \st \Im z > 0 }$ is the strict upper half plane. 
\item Harder--Narasimhan filtrations with respect to $\sigma$ exist for objects in $\cA$, where recall that for $E \in \cA$ the slope is defined by 
\begin{equation*}
\mu_Z(E) = \begin{cases}
- \frac{\Re Z(v(E))}{\Im Z(v(E))} & \text{if } \Im Z(v(E)) > 0 , \\ 
+ \infty & \text{otherwise}, 
\end{cases}
\end{equation*} 
$E$ is called $\sigma$-(semi)stable if $\mu_Z(F) (\leq) < \mu_Z(E/F)$ for every proper subobject $0 \neq F \subset E$, and a Harder--Narasimhan filtration of $E$ with respect to $\sigma$ is a filtration with $\sigma$-semistable factors of decreasing slope. 
\end{enumerate}
If in condition~\eqref{stability-function} we instead require the weaker condition $Z(v(E)) \in \bH \sqcup \bR_{\leq 0}$, then $\sigma$ is called \emph{weak pre-stability condition on} $\cD$.

If $\Lambda = \rK_0(\cD)$ we simply say $\sigma$ is a pre-stability condition. 
If $\cD$ is a proper over a field $k$ (i.e. $\sum \dim_k \Hom(E,F[i]) < \infty$ for all $E, F \in \cD$), then the numerical Grothendieck group 
$\Knum(\cD)$ is defined as the quotient of $\rK_0(\cD)$ by the kernel of the Euler paring 
\begin{equation*}
\chi(E,F) = \Sigma_i (-1)^i \dim_k \Ext^i(E,F), 
\end{equation*} 
and we say $\sigma$ is \emph{numerical} if 
$v \colon \rK_0(\cD) \to \Lambda$ 
factors via $\rK_0(\cD) \to \Knum(\cD)$. 

Assume $\Lambda$ is a free abelian group of finite rank and $\|-\|$ is a norm on $\Lambda \otimes \bR$. 
Then we say $\sigma$ satisfies the \emph{support property} 
if there exists a constant $C > 0$ such that for all $\sigma$-semistable objects $0 \neq E \in \cA$ we have 
\begin{equation*}
\| v(E) \| \leq C |Z(v(E))|. 
\end{equation*} 
We say $\sigma$ is a \emph{stability condition with respect to $\Lambda$} if it is a pre-stability condition satisfying the support property. 
When $\cD$ is proper over a field, we say $\sigma$ is a \emph{full numerical stability condition} if it is a stability condition with respect to $\Knum(\cD)$. 
\end{definition} 

\begin{remark}
By abuse of notation, we shall sometimes conflate $Z \circ v$ and $Z$. 
\end{remark}

We denote by $\Stab_{\Lambda}(\cD)$ the set of stability conditions with respect to $\Lambda$. Recall that this set naturally has the structure of a complex manifold, such that the map 
\begin{equation*}
\cZ \colon \Stab_{\Lambda}(\cD) \to \Hom(\Lambda, \bC), \quad (\cA, Z) \mapsto Z 
\end{equation*} 
is a local isomorphism \cite{bridgeland}. 

Below we will need to descend stability conditions which are invariant under a group action to the quotient by the group action. 
For simplicity, we explain this in the geometric case where $\cD$ is the derived category of a variety, but the results easily extend to the noncommutative case, as explained in \cite[Theorem 4.8]{enriques-categories}. 
Thus let $X$ be a variety equipped with an action by a finite group $G$, 
let $\Lambda$ be a free abelian group of finite rank equipped with a $G$-action, 
and let $v \colon \rK_0(X) \to \Lambda$ be a $G$-equivariant homomorphism. 
There is an induced $G$-action on $\Stab_{\Lambda}(X) \coloneqq \Stab_{\Lambda}(\Db(X))$; 
explicitly, the action of $g \in G$ on $\sigma = (\cA, Z) \in \Stab_{\Lambda}(X)$ is 
$g \cdot \sigma = ( g \cdot \cA, g \cdot Z)$ where $g \cdot \cA = g_* \cA$ and $g \cdot Z = Z \circ g^*$. 
In particular, we can consider the collection of $G$-invariant stability conditions, i.e. the fixed locus for this action: 
\begin{equation*}
\Stab_{\Lambda}(X)^{G} = \set{ \sigma \in \Stab_{\Lambda}(X) \st g \cdot \sigma = \sigma \text{ for all } g \in G }. 
\end{equation*} 
This locus is related to the space of stability conditions on the quotient stack $[X/G]$ as follows. 
Let $\pi \colon X \to [X/G]$ be the quotient morphism. 
We define the homomorphism 
\begin{equation*}
v^G = v \circ \pi^* \colon \rK_0([X/G]) \to \Lambda. 
\end{equation*} 
For $\sigma = (\cA, Z) \in \Stab_{\Lambda}(X)^G$, we define 
\begin{equation*}
\cA^G = \set{E \in \Db([X/G]) \st \pi^*E \in \cA } , \quad 
Z^G = Z \circ v^G, \quad 
\text{and} \quad 
\sigma^{G} = (\cA^G, Z^G).  
\end{equation*} 

\begin{theorem}[\cite{polishchuk, inducing-stability}]
\label{theorem-stability-G}
The locus $\Stab_{\Lambda}(X)^{G} \subset \Stab_{\Lambda}(X) $ is a union of connected 
components. 
Further, the above construction preserves the property that a stability condition is numerical, and induces a closed embedding 
\begin{equation*} 
(-)^G \colon \Stab_{\Lambda}(X)^{G} \to \Stab_{\Lambda}([X/G]) 
\end{equation*} 
with image 
\begin{equation*}
\Stab_{\Lambda}([X/G])^{\circ} = 
\set{ \bar{\sigma} = (\bar{\cA}, \bar{Z})   \in \Stab_{\Lambda}([X/G]) \st \pi_* \cO_X \otimes \bar{\cA} \subset \bar{\cA} }, 
\end{equation*}  
which is a union of connected components of $\Stab_{\Lambda}([X/G])$. 
Finally, for $\sigma \in \Stab_{\Lambda}(X)^{G}$, an object $E \in \Db([X/G])$ is 
$\sigma^G$-semistable if and only if $\pi^*E$ is $\sigma$-semistable. 
\end{theorem}

\begin{proof}
At the set-theoretic level, the bijection between $\Stab_{\Lambda}(X)^{G}$ and $\Stab_{\Lambda}([X/G])^{\circ}$ 
is \cite[Proposition 2.2.3]{polishchuk}. 
The finer statements about the stability manifolds are \cite[Theorem 1.1]{inducing-stability}. 
To be more precise, our statement of the result differs slightly from the quoted ones, 
because we consider stability conditions satisfying the support property with respect to $\Lambda$; 
however, the arguments of \cite{polishchuk, inducing-stability} go through in our context, 
cf. \cite[Theorem 10.1]{CY3-stability}. 
\end{proof}

\begin{remark}
There is a similar statement for pre-stability conditions: the construction 
$(-)^G$ gives a bijection between the set of $G$-invariant pre-stability conditions $\sigma$ on $\Db(X)$ 
and the set of pre-stability conditions $\bar{\sigma}$ on $\Db([X/G])$ for which tensoring by 
$\pi_* \cO_X$ is $t$-exact. 
\end{remark} 

\subsection{Products with curves}
\label{section-products-with-curves} 
In \cite{products}, Liu proves the following result for stability conditions on the product of a variety with a curve. 
We say a pre-stability condition $\sigma = (\cA, Z)$ is \emph{rational} if the image of $Z$ lies in $\bQ + \bQ i$.  

\begin{theorem}[\cite{products}]
\label{theorem-product}
Let $X$ be a smooth projective variety,  
let $\sigma = (\cA, Z)$ be a rational stability condition on $X$,  
let $C$ be a smooth projective curve with an ample line bundle $L$, 
and let $s, t \in \bQ_{> 0}$. 
Then there is a rational stability condition 
$\sigma_{C}^{L, s,t} = (\cA_{C}^{L, t}, Z_C^{L,s,t})$ on $X \times C$, 
which is numerical if $\sigma$ is so.  
\end{theorem} 
 
We will recall the construction of $\sigma_C^{L,s,t}$ more precisely below. 
The starting point is the following result. 

\begin{theorem}
\label{theorem-ZSL}
Let $X$ be a variety and 
let $\sigma = (\cA, Z)$ be a stability condition on $X$ whose heart $\cA$ is noetherian. 
Let $S$ be a smooth projective variety with an ample line bundle $L$. 
Let $p \colon X \times S \to X$ and $q \colon X \times S \to S$ be the projections. 
\begin{enumerate}
\item \label{AS}
There is a noetherian heart of a bounded t-structure on $\Db(X \times S)$ given by 
\begin{equation*} 
\cA_S = \set{E \in \Db(X \times S) \st p_*(E \otimes q^*L^{n}) \in \cA \text{ for } n \gg 0 } , 
\end{equation*} 
which does not depend on the choice of $L$. 
\item \label{ZSL} 
For $E \in \rK_0(X \times S)$ there is a polynomial $Z_S(E) \in \bC[x]$ of degree at most $\dim(S)$ 
given by 
\begin{equation*}
Z^L_S(E)(n) = Z(p_*(E \otimes q^*L^{n})) 
\end{equation*} 
for $n \in \bZ$, whose leading coefficient gives a weak pre-stability condition with heart $\cA_S$. 
If $X$ is smooth and proper and $\sigma$ is numerical, 
then $Z^L_S$ only depends on the classes of $E$ in $\Knum(X \times S)$ and of $L$ in $\Knum(S)$. 
\end{enumerate}
\end{theorem}

\begin{proof}
\eqref{AS} is \cite{AP, polishchuk}. 
The first sentence of~\eqref{ZSL} is \cite[Theorem 3.3]{products}, 
while the second follows directly from the definition.   
\end{proof}

\begin{construction}
Using the preceding result, let us review the construction of the stability condition $\sigma_C^{L,s,t}$ from Theorem~\ref{theorem-product}. 
For $E \in \rK_0(X \times C)$, the polynomial $Z_C^L(E)(n)$ is linear and hence may be written as 
\begin{equation*}
Z_C^L(E)(n) = a(E)n + b(E) + i(c(E)n + d(E))
\end{equation*}
for homomorphisms $a,b,c,d \colon \rK_0(X \times C) \to \bR$. 
Define $Z_C^{L,t} \colon \rK_0(X \times C) \to \bC$ by 
\begin{equation*}
Z_C^{L,t}(E) = a(E)t - d(E) + ic(E)t. 
\end{equation*} 
Then the pair $\sigma_t = (\cA_C, Z_{C}^{L,t})$ is a weak pre-stability condition \cite[Lemma 4.1]{products}, so we may tilt at slope $0$ to obtain a new heart $\cA_{C}^{L,t}$. 
Define $Z_C^{L,s,t} \colon \rK_0(X \times C) \to \bC$ as 
\begin{equation*}
Z_C^{L,s,t}(E) = c(E)s + b(E) + i(-a(E)t + d(E)). 
\end{equation*} 
Then $\sigma_{C}^{L,s,t} = (\cA_C^{L, t}, Z_C^{L,s,t})$ is a stability condition by \cite[Theorem 4.7 and Theorem~5.9]{products}. 
To be more precise, suppose that $\sigma$ is a stability condition with respect to a homomorphism $v \colon \rK_0(X) \to \Lambda$. 
Then by \cite[Lemma 5.2]{products}, the central charge $Z_C^{L,s,t}$ factors as 
\begin{equation*}
\begin{tikzcd}[ampersand replacement=\&]
Z_C^{L,s,t} \colon \rK_0(X \times C) 
 \arrow{r}{ \begin{pmatrix} v^L_1 \\ v^L_2 \end{pmatrix} } \& \Lambda \oplus \Lambda  \arrow{rrr}{\begin{pmatrix} s \Im Z - i t \Re Z & Z \end{pmatrix} } \& \& \& \bC
\end{tikzcd}
\end{equation*} 
where $v^L_1, v^L_2 \colon \rK_0(X) \to \Lambda$ are homomorphisms given by 
\begin{align*}
v^L_1(E) & = v(p_*(E \otimes q^* L^n)) - v(p_*(E \otimes q^*L^{n-1})) , \\ 
v^L_2(E) & = v(p_*(E \otimes q^* L^n)) - n \cdot v^L_1(E), 
\end{align*} 
for $n \gg 0$. 
Therefore, writing $\Lambda_C = \Lambda \oplus \Lambda/\ker(Z)$ and $v^L_C \colon \rK_0(X \times C) \to \Lambda_C$ for the composition of $\begin{pmatrix} v^L_1 \\ v^L_2 \end{pmatrix}$ with the projection to $\Lambda_C$, we also obtain a factorization 
\begin{equation*}
\begin{tikzcd}[ampersand replacement=\&]
Z_C^{L,s,t} \colon \rK_0(X \times C) 
 \arrow{r}{ v_C^L } \& \Lambda_C \arrow{rrr}{\begin{pmatrix} s \Im Z - i t \Re Z & Z \end{pmatrix} } \& \& \& \bC. 
\end{tikzcd}
\end{equation*} 
By \cite[Lemma 5.7 and Remark 5.8]{products}, 
$\sigma_C^{L,s,t}$ is a stability condition with respect to $v^L_C$. 
Note that by construction, if $\sigma$ is numerical then so is $\sigma_C^{L,s,t}$. 
\end{construction} 

Next we observe that the above construction behaves well with respect to group actions. 
Note that if $\Gamma$ is a group acting on $X \times S$, then there is an induced action on the set 
of homomorphisms $W \colon \rK_0(X \times S) \to \bC[x]$; namely, $\gamma \in \Gamma$ acts on $W$ 
by $\gamma \cdot W = W \circ \gamma^*$. 

\begin{lemma}
\label{lemma-ZSL-equivariant}
In the situation of Theorem~\ref{theorem-ZSL}, 
assume that $G$ is a group acting on $X$, 
$H$ is a group acting on $S$, 
$\Gamma$ is a group acting on $X \times S$, 
and $\phi \colon \Gamma \to G$ and $\psi \colon \Gamma \to H$ are 
homomorphisms such that the projections $p \colon X \times S \to X$ and 
$q \colon X \times S \to S$ are equivariant for the group actions. 
Let $\gamma \in \Gamma$ and set $g = \phi(\gamma) \in G$ and $h = \psi(\gamma) \in H$. 
\begin{enumerate}
\item We have $\gamma \cdot \cA_S = (g \cdot \cA)_S$, where the right side denotes the construction of Theorem~\ref{theorem-ZSL}\eqref{AS} applied to the heart $g \cdot \cA$. 
\item We have 
$\gamma \cdot Z_S^{L} = (g \cdot Z)_{S}^{h_*L}$, 
where the right side denotes the construction of Theorem~\ref{theorem-ZSL}\eqref{ZSL} applied to the stability condition 
$g \cdot \sigma = (g \cdot \cA, g \cdot Z)$ with the line bundle $h_*L$. 
\end{enumerate}
In particular, if $\sigma$ is $G$-invariant and $L$ is $H$-invariant 
(or if the class of $L$ in $\Knum(S)$ is $H$-invariant in case $\sigma$ is numerical), 
then $\cA_S$ and $Z_S^{L}$ are $\Gamma$-invariant. 
\end{lemma} 

\begin{proof}
For $E \in \Db(X \times S)$ and $n \in \bZ$, we have
\begin{align*} 
p_*(\gamma^*E \otimes q^*L^n) & \simeq p_*(\gamma^*E \otimes \gamma^* \gamma_* q^* L^n) \\ 
& \simeq p_*(\gamma^*(E \otimes q^* h_*(L)^n)) \\ 
& \simeq  g^* p_* (E \otimes q^* h_*(L)^n) 
\end{align*} 
where the second and third lines hold by the equivariance of $q$ and $p$. 
The result follows. 
\end{proof} 

\begin{lemma}
\label{lemma-product-invariant}
In the situation of Theorem~\ref{theorem-product}, let $v \colon \rK_0(X) \to \Lambda$ be the homomorphism with respect to which $\sigma$ is a stability condition. 
Assume that 
$G$ is a group acting on $X$ and $\Lambda$ such that 
$v \colon \rK_0(X) \to \Lambda$ is $G$-equivariant and $\sigma$ is $G$-invariant, 
$H$ is a group acting on $C$, either $\sigma$ is numerical or $L$ is $H$-invariant, 
$\Gamma$ is a group acting on $X \times C$, 
and $\phi \colon \Gamma \to G$ and $\psi \colon \Gamma \to H$ are 
homomorphisms such that the projections $p \colon X \times C \to X$ and 
$q \colon X \times C \to C$ are equivariant. 
\begin{enumerate}
\item \label{product-invariant-vC}
The homomorphism $v^L_C \colon \rK_0(X \times C) \to \Lambda_C$ is $\Gamma$-equivariant, 
where $\Gamma$ acts on $\Lambda_{C}$ via the homomorphism 
$\phi \colon \Gamma \to G$ and the diagonal action of $G$ on $\Lambda \oplus \Lambda$. 
\item \label{product-invariant-sigmaC}
For $s,t \in \bQ_{> 0}$, 
the stability condition $\sigma_{C}^{L, s,t} = (\cA_{C}^{L, t}, Z_C^{L,s,t})$ is $\Gamma$-invariant. 
\end{enumerate} 
\end{lemma}

\begin{proof}
Note that since $C$ is a smooth projective curve, automorphisms of $C$ act trivially on $\Knum(C)$, 
so the class of $L$ in $\Knum(C)$ is automatically $H$-invariant. 
Then by the same argument as in the proof of Lemma~\ref{lemma-ZSL-equivariant}, we obtain the claim~\eqref{product-invariant-vC}.

The heart $\cA_C^{L,t}$ is constructed by tilting from the pair $\sigma_t = (\cA_C, Z_C^{L,t})$, and both $Z_C^{L,t}$ and $Z_C^{L,s,t}$ are constructed from the coefficients of $Z_C^L$; thus, the $\Gamma$-invariance of $\sigma_{C}^{L, s,t}$ follows from Lemma~\ref{lemma-ZSL-equivariant}. 
\end{proof}

Let us note that in many cases where stability conditions are known to exist, it is easy to arrange that they are fixed by a group action. 

\begin{example}
\label{example-stability-curve}
Let $C$ be a smooth  projective curve. 
Then the standard full numerical stability condition 
$\sigma = (\Coh(C), - \deg + i \, \rank )$ 
is fixed under automorphisms of~$C$. 
Indeed, automorphisms of $C$ act trivially on $\Knum(C)$. 
\end{example} 

\begin{example}
\label{example-invariant-stability-surface}
Let $X$ be a smooth projective surface,  
let $\omega$ be a real ample divisor on $X$, and let $D$ be a real divisor on $X$. 
Let $\ch^D= e^{-D}\ch$ be the twisted Chern character, where $e^{-D}$ is the formal exponential of the divisor class $-D$. 
Define $\Coh^{\omega, D}(X) \subset \Db(X)$ to be the heart of a bounded t-structure obtained by tilting the weak pre-stability condition $(\Coh(X), -\omega \ch_1^D + i \omega^2 \ch_0^D)$ at slope $0$, and 
define $Z_{\omega, D} \colon \rK_0(X) \to \bC$ by 
\begin{equation*}
Z_{\omega,D} = \left( -\ch_2^D + \frac{\omega^2}{2} \ch_0^D\right) +  i \omega \ch_1^D.
\end{equation*}
Then by \cite{bridgeland-K3, arcara-bertram} (see also \cite[\S6]{macri-schmidt} for an exposition), 
$\sigma_{\omega, D} = (\Coh^{\omega, D}(X), Z_{\omega, D})$ is a full numerical stability condition on $X$, 
which is rational when $\omega$ and $D$ are rational. 

Now suppose that a finite group $G$ acts on $X$. 
If the classes of $\omega$ and $D$ in the N\'{e}ron--Severi group are fixed by $G$, then 
$\sigma_{\omega, D}$ is also fixed by $G$. 
Indeed, in this case $G$ fixes the homomorphism $-\omega \ch_1^D + i \omega^2 \ch_0^D$, hence also the heart $\Coh^{\omega, D}(X)$, and $G$ also fixes the central charge $Z_{\omega, D}$. 
Note that if $\omega$ and $D$ are not fixed by $G$, then we can always replace them with 
$\sum_{g \in G} g^*(\omega)$ and $\sum_{g \in G} g^*(D)$ to get $G$-invariant classes. 
\end{example} 

\subsection{Proof of Theorem~\ref{theorem-stability-conditions}}
Combining the results from \S\ref{section-stability-group-actions}-\S\ref{section-products-with-curves} gives the key remaining ingredient for Theorem~\ref{theorem-stability-conditions}. 

\begin{proposition}
\label{proposition-Ci}
Let $X$ be a smooth projective variety with an action by a finite group $G$. 
Let $\sigma$ be a $G$-invariant rational stability condition on $X$. 
For $1 \leq i \leq n$, let $C_i$ be a smooth projective curve with an action by a finite group $G_i$. 
\begin{enumerate}
\item \label{Ci-Gi}
There exists a $(G \times G_1 \times \cdots \times G_n)$-invariant rational stability condition on 
$X \times C_1 \times \cdots \times C_n$, which is numerical if $\sigma$ is so. 
\item \label{Ci-Gamma}
For any subgroup $\Gamma \subset \prod_{i = 1}^n G_i$, there exists a 
stability condition on $[(X \times C_1 \times \cdots \times C_n)/ \Gamma]$, 
which is numerical if $\sigma$ is so. 
\end{enumerate}
\end{proposition}

\begin{proof}
The claim \eqref{Ci-Gi} follows by applying Theorem~\ref{theorem-product} and Lemma~\ref{lemma-product-invariant} 
one factor of $C_i$ at a time, while 
\eqref{Ci-Gamma} follows from~\eqref{Ci-Gi} and Theorem~\ref{theorem-stability-G}. 
\end{proof}

In the setup of Theorem~\ref{theorem-stability-conditions}, 
Proposition~\ref{proposition-Ci} gives a stability condition on the quotient stack 
$[(X \times C_1 \times \cdots \times C_n)/ G_{n+1}]$, which is numerical if $\sigma$ is so. 
Since this quotient stack is derived equivalent to $\cyclicquotient(X, C_1, \dots, C_n)$ by Theorem~\ref{theorem-derived-equivalence}, we obtain a stability condition $\tilde{\sigma}$ on 
$\cyclicquotient(X, C_1, \dots, C_n)$. 
If the Albanese morphism of $X$ is finite and all $C_i$ have genus at least $1$, 
then \cite[Theorem 1.1]{FLZ} shows that for any numerical stability condition on 
$X \times C_1 \times \cdots C_n$, all skyscraper sheaves of points are stable of the same phase, say $\phi$. 
By Remark \ref{remark-Phi-skyscraper}, the composition of the equivalence with the pullback functor 
\begin{equation*}
\Dperf(\cyclicquotient(X, C_1, \dots, C_n)) \simeq 
\Dperf([(X \times C_1 \times \cdots \times C_n)/ G_{n+1}]) \to \Dperf(X \times C_1 \times \cdots \times C_n)
\end{equation*}
sends any skyscraper sheaf of a point to an object in the extension closure of skyscraper sheaves of points, which is in particular a semistable object of phase $\phi$ by the previous sentence. 
Now part~\eqref{stability-conditions-semi-geometric} of Theorem~\ref{theorem-stability-conditions} follows from the last statement of Theorem~\ref{theorem-stability-G}. \qed

%%%%%%%%%%%%%%%%%%%%%%%%%%%%%%%%%%%%%%%%%%%%%%%%%%%%%%

\newcommand{\etalchar}[1]{$^{#1}$}
\providecommand{\bysame}{\leavevmode\hbox to3em{\hrulefill}\thinspace}
\providecommand{\MR}{\relax\ifhmode\unskip\space\fi MR }
\providecommand{\MRhref}[2]{%
  \href{http://www.ams.org/mathscinet-getitem?mr=#1}{#2}
}
\providecommand{\href}[2]{#2}

%%%%%%%%%%%%%%%%%%%%%%%%%%%%%%%%%%%%%%%%%%%%%%%%%%%%%%

\end{document}